\documentclass[12pt]{amsart}
\usepackage{amsmath}
\usepackage{amscd}
\usepackage[all]{xy}

\newcommand{\Pic}{\operatorname{Pic}}
\newcommand{\rat}{{\operatorname{rat}}}

\newcommand{\Alb}{\operatorname{Alb}}
   
\newcommand{\End}{\operatorname{End}}

\newcommand{\Gammaf}{\Gamma_{\kern-3pt f}}

\def\alg{\text{alg}}
\def\tr{\text{tr}} 
\def\lra{\longrightarrow}
\def\map#1{{\buildrel #1 \over \lra}}

\newcommand{\C}{\mathbb{C}}

\renewcommand{\P}{\mathbb{P}}
\renewcommand{\L}{\mathbb{L}}
\newcommand{\Q}{\mathbb{Q}}
\newcommand{\Z}{\mathbb{Z}}

\newcommand{\sM}{\mathcal{M}}
\newcommand{\sO}{\mathcal{O}}

\newcommand{\Mrat}{\sM_\rat}

\numberwithin{equation}{section}

\theoremstyle{plain}
\newtheorem{thm}[equation]{Theorem}
\newtheorem{prop}[equation]{Proposition}
\newtheorem{lem}[equation]{Lemma}
\newtheorem{cor}[equation]{Corollary}

\theoremstyle{definition}
\newtheorem{defn}[equation]{Definition}
\newtheorem{ex}[equation]{Example}
\newtheorem{rk}[equation]{Remark}
\newtheorem{constr}[equation]{Construction}
\newtheorem{substuff}{Remark}[equation]

\newtheorem{subrem}[substuff]{Remark}
\begin{document}
 
\title[\ ] 
{Some surfaces of general  type for which Bloch's  conjecture holds}

\author{C. Pedrini and C. Weibel}
\date{\today}
\bigskip

\begin{abstract} We give many examples of surfaces of general type
with $p_g=0$ for which Bloch's conjecture holds, for all values of
$K_2\ne9$. Our surfaces are equipped with an involution.
\end{abstract}
\maketitle

Let $S$ be a smooth complex projective surface with $p_g(S)=0$.
Bloch's conjecture states that the Albanese map $A_0(S)_0\to\Alb(S)$
is an isomorphism, where $A_0(S)_0$ is the Chow group of 0-cycles 
of degree 0 on $S$.
It is known for all surfaces except those of general
type (see \cite{BKL}). For a surface $S$ of general type with 
$p_g(S)=0$ we also have $q(S)=0$, i.e., $\Alb(S)=0$ and the canonical divisor 
satisfies $1 \le K^2 \le 9$.

In the decades since this conjecture was formulated, surfaces of general
type have become somewhat better understood. Two key developments have been 
(i) the results of S.\,Kimura on finite dimensional motives in \cite{Ki} 
and (ii) the notion of the {\it transcendental motive $t_2(S)$} 
which was introduced in \cite{KMP}. 
This includes the theorem that if $S$ is a surface with $p_g(S)=q(S)=0$ 
then Bloch's conjecture holds for $S$ iff $t_2(S)=0$; 
see Lemma \ref{t2=0.BC}.

In this paper we give motivic proofs of Bloch's conjecture for 
several examples of surfaces of general type 
for each value of $K^2$ between 1 and 8. 
This includes some numerical Godeaux surfaces, classical Campedelli 
surfaces, Keum-Naie surfaces, Burniat surfaces and Inoue's surfaces. 
All these surfaces carry an involution. 
We can say nothing about the remaining case $K^2\!=\!9$,
because a surface of general type with $p_g\!=\!0$ and $K^2\!=\!9$ has no
involution (\cite[2.3]{DMP}).

Bloch's conjecture is satisfied by all surfaces whose 
minimal models arise as quotients $C_1\times C_2/G$ of the product of 
two curves of genera $\ge\!2$ by the action of a finite group $G$. 
A complete classification of these surfaces has been
given in \cite{BCGP} and \cite[0.1]{BCG}; the special case where
$G$ acts freely only occurs when $K^2_S=8$. We also show 
in Corollary \ref{K2=8rational} that Bloch's conjecture holds for surfaces 
with an involution $\sigma$ for which $K^2=8$ and $S/\sigma$ is rational.
The only known examples with $K^2=7$ are Inoue's surfaces, which satisfy
Bloch's conjecture, and a new family, recently constructed by Y.\,Chen 
in \cite{Ch}, who also shows that these surfaces satisfy Bloch's conjecture. 
Burniat surfaces with $K^2=6$ satisfy Bloch's conjecture by Theorem 
\ref{thm:Burniat}; this provides also examples with $K^2=3,4,5$. 
Propositions \ref{K2=6.1} and \ref{K2=6.2} give examples of surfaces 
with $K^2=6$ which are not Burniat, and which satisfy Bloch's conjecture.
Other examples with $K^2=3,4$ are given in Section 6. Some examples 
with $K^2=1,2$ are treated in Sections \ref{sec:K2=1} and \ref{sec:K2=2}.

\medskip
{\bf Notation.} 
We will work in the (covariant) category $\Mrat(k)$ of Chow motives with
coefficients in $\Q$, where the morphisms from $h(X)$ to $h(Y)$ are just 
the elements of $A^d(X\times Y)$, $d=\dim(Y)$. (The graph $\Gammaf$
of a morphism $f:X\to Y$ determines a map $f^t:h(X)\to h(Y)$.)
Here, and in the rest of this paper, $A^i(X \times Y)$ denotes the Chow
group $CH^i(X \times Y)\otimes \Q$ of codimension $i$ cycles modulo
rational equivalence, with $\Q$ coefficients.

\medskip
\section{Preliminaries}

Let $f: X\to Y$ be a finite morphism between smooth projective varieties
of dimension $d$. Then $\Gamma^t_f$ determines a map $h(Y)\to h(X)$
and the composition $\Gammaf\circ{}^t\Gammaf$ is $\deg(f)$ times the
identity of $h(Y)$.

\begin{lem}\label{Y=pX}
$f: X\to Y$ be a finite morphism between smooth projective varieties
of dimension $d$, and set $p=({}^t\Gammaf\circ\Gammaf)/\deg(f)$.
Then $p:h(X)\to h(X)$ is idempotent, and expresses $h(Y)$ as a direct
summand of $h(X)$ in $\Mrat(k)$.
\end{lem} 

\begin{proof}
We compute: $p\circ p=
{}^t\Gammaf^t\circ(\Gammaf\circ{}^t\Gammaf)\circ\Gammaf/\deg(f)=p$.
\end{proof}

\begin{ex}\label{ex:Y=Xsigma}
Suppose that $\sigma$ is an involution on $X$ and $Y=X/\sigma$.
Then $p=(1+\sigma)/2$ is represented by the correspondence
$(\Delta_X+\Gamma_\sigma)/2$, because 
${}^t\Gammaf\circ\Gammaf)=\Delta_X+\Gamma_\sigma$.
In particular, $h(Y)=h(X)^\sigma$.
\end{ex}

Let $\rho(S)$ denote the rank of $NS(S)$.
The Riemann-Roch Theorem gives a well known formula for $\rho$:

\begin{lem}\label{rho+K2}
If $p_g(S)=q(S)=0$, then $\rho(S)=10-K^2_S$.
\end{lem}

\begin{proof}
The topological Euler characteristic of $S$ is $\deg(c_2)=2 +\rho(S)$.
Since the Euler characteristic of $\sO_S$ is 1, the Riemann-Roch Theorem
(\cite[15.2.2]{Fu}) yields $12=K^2_S +\deg(c_2)$.
Thus $\rho(S)= 10-K^2_S$.
\end{proof}

\begin{subrem} 
If $S$ is a minimal surface of general type, then $K^2_S >0$
\cite[VII.2.2]{BHPV}. Since $\rho(S) \ge1$, we derive the inequality
$1\le K^2_S\le9$. We also have $\deg c_2 >0$, see
\cite[VII.1.1]{BHPV}. By Noether's formula $1-q(S)+p_g(S)$ is
positive, hence if $p_g(S)=0$ then also $q(S)=0.$
\end{subrem} 

\medskip
\paragraph{\bf The algebraic motive $h^\alg_2(S)$}
If $S$ is a surface, the Neron-Severi group $NS(S)$ determines a summand
of the motive of $S$ (with coefficients $\Q$). To construct it,
choose an orthogonal basis $\{E_1,\cdots,E_{\rho}\}$ for the 
$\Q$-vector space $NS(S)_{\Q}$, where the self-intersections $E_i^2$
are nonzero. ($\rho$ is the rank of $NS(S)$.) Then
the correspondences $\epsilon_i=\frac{[E_i\times E_i]}{(E_i)^2}$
are orthogonal and idempotent, so
\[
\pi^\alg_2(S)=\sum_{1\le i\le\rho} \frac{[E_i\times E_i]}{(E_i)^2}
\]
is also an idempotent correspondence. 
Since $\{ E_i/(E_i)^2\}$ is a dual basis to the $\{ E_i\}$, it follows from
\cite[7.2.2]{KMP} that $\pi^\alg_2(S)$ is independent of the choice of basis.
We set $M_i=(S,\epsilon_i, 0)$
and $h^\alg_2(S)=(S,\pi^\alg_2(S))=\oplus M_i$.
In fact, $M_i\simeq\L$ for all $i$ by \cite[7.2.3]{KMP}, 
so we have isomorphisms
$h^\alg_2(S) \cong \L^{\oplus\rho}$ and $H^2(h^\alg_2(S))\cong NS(S)_\Q$.

\smallskip
\paragraph{\bf The transcendental motive $t_2$}
We also need a description of the transcendental motive $t_2$ of a 
surface $S$. It is well known that the motive $h(S)$ has a 
Chow-K\"unneth decomposition as $\sum_0^4 h_i(S)$,
where $H^i(S)=H^i(h_i(S))$. The middle factor $h_2(S)$ further
decomposes as $h_2(S)=h^\alg_2(S)\oplus t_2(S)$; see \cite{KMP}. 

The factor $t_2(S)=(S,\pi^\tr_2,0)$ is called the {\it transcendental part} 
of the motive (cf.\,\cite[7.2.3]{KMP}).  This terminology is justified 
by the following result, which identifies $H^2(t_2(S))=\pi^{tr}_2H^2(S)$ 
with the Hodge-theoretic group $H^2_\tr(S)$. 

\begin{lem}\label {int-pairing} 
Under the intersection pairing on $H^2(S,\C)$ the orthogonal complement 
$H^2_{tr}(S)$ of $NS(S)\otimes\C$ is $\pi^{tr}_2 H^2(S)$.
\end{lem}

\begin{proof} 
Given $a \in H^2_{tr}(S)$ and $b \in NS(S)$ we have
$$
a\cdot b=(a\cdot\pi^{tr}_2)\cdot (\pi^{alg}_2\cdot b)= 
a\cdot (\pi^{tr}_2\cdot \pi^{alg}_2)\cdot b=0,
$$
since $\pi^{tr}_2 \cdot \pi^{alg}_2=0$ by \cite[p.\,289]{CG}.
\end{proof}

The following result was established in \cite[7.4.9 \& 7.6.11]{KMP}.

\begin{lem}\label{t2=0.BC}
If $S$ is a smooth projective complex surface, then
the following are equivalent:\\
(a) $t_2(S)=0$ holds; \\
(b) $p_g(S)=0$ and the motive $h(S)$ is finite dimensional; \\
(c) $p_g(S)=0$ and  $S$ satisfies Bloch's conjecture.
\end{lem}

We will often use without comment the fact that $t_2$ is a birationally
invariant functor on the category of smooth 2-dimensional varieties;
this is proven in \cite[7.8.11]{KMP}.

\begin{lem}\label{lem:t2^G}
If $\sigma$ is an involution of a surface $S$, and $Y$ is a 
desingularization of $S/\sigma$, then $t_2(Y)=t_2(S)^{\sigma}$.
\end{lem}

\begin{proof} 
Because $t_2$ is a birational invariant, we may blow up $S$ to assume
that $Y=S/\sigma$. (See \eqref{square:S-sigma} below). Since 
$h(S)\to h(Y)$ sends $t_2(S)$ to $t_2(Y)$, and $h(Y)=h(S)^{\sigma}$ by 
Example \ref{ex:Y=Xsigma}, the result follows. 
\end{proof}

\goodbreak
\section{Involutions on surfaces}\label{sec:involutions}

Let $S$ be a smooth projective surface with an involution $\sigma$.
The fixed locus consists of a 1-dimensional part $D$ (a union of
smooth curves, possibly empty) and 
$k\ge0$ isolated fixed points $\{P_1,\cdots, P_k \}$.
The images $Q_i$ of the $P_i$ are nodes on the quotient surface $S/\sigma$,
and $S/\sigma$ is smooth elsewhere. To resolve these singularities,
let $X$ denote the blow-up of $S$ at the set of isolated fixed points;
$\sigma$ lifts to an involution on $X$ (which we will still call $\sigma$),
and the quotient $Y=X/\sigma$ is a desingularization of $S/\sigma$. 
The images $C_1,\dots,C_k$ in $Y$ of the exceptional divisors of $X$
are disjoint nodal curves, 
i.e., smooth rational curves with self-intersection $-2$. 
In summary, we have a commutative diagram

\begin{equation}\label{square:S-sigma}
\CD X @>{h}>>S \\
@VV{\pi}V  @VV{f}V \\
Y@>{g}>> S/\sigma. \endCD
\end{equation}

\smallskip
The image $f_*D$ is a smooth curve on $S/\sigma$, disjoint from
the singular points $Q_i$, and its proper transform $B'=g^*(f_*D)$
is smooth and disjoint from the curves $C_i$. It follows that 
$\pi:X\to Y$ is a double cover with smooth branch locus $B=B'+\sum C_i$.
As such, $\pi$ is determined by a line bundle $L$ on $Y$ such that
$2L \equiv B.$

\begin{lem}\label{2L2=D2-k}
$2 L^2= D^2-k$.
\end{lem}

\begin{proof}
Because the curves $C_i$ on $X$ have $C_i^2=-2$ and are disjoint
from $B'=g^*(D)$ we have $(B')^2=2D^2$. Since $2L \equiv B'+\sum C_i$
we have $4 L^2= (B')^2+\sum C_i^2= 2D^2 -2k$.
\end{proof}

\begin{subrem}\label{l(2K+L)}
If $S$ is a minimal surface of general type and $p_g(S)=0$, 
it is proven in \cite[3.3]{CCM} that $k\ge4$ and the linear system
$2K_Y+L$ has dimension $l(2K_Y+L)=(K_S^2+4-k)/2$.
\end{subrem}

Since $t_2(-)$ is a birational invariant for smooth projective surfaces,
the maps $h: X \to S$ and $\pi: X \to Y$ induce a morphism
\[
\theta: t_2(S) \cong t_2(X) \to t_2(Y).
\]
By Lemma \ref{Y=pX} and Example \ref{ex:Y=Xsigma}, 
$\theta$ is the projection onto the direct summand
$t_2(S)^\sigma$ of $t_2(S)$, and $A_0(Y)$ is the direct summand
of $A_0(X)$ fixed by $\sigma$.

\goodbreak
\begin{prop}\label{old nice thm}
Let $S$ be a smooth surface $q(S)\!=\!0$.
If $\sigma$ is an involution on $S$ then:
\begin{enumerate}
\item[(i)] $t_2(S)\cong t_2(Y) \Longleftrightarrow \bar \sigma=+1 
   \text{ in } \End_{\Mrat}(t_2(S))$
 \item[(ii)] $t_2(Y)=0 \Longleftrightarrow  \bar\sigma= -1 
   \text{ in } \End_{\Mrat}(t_2(S)).$
\end{enumerate}
Here $\bar\sigma$ is the endomorphism of $t_2(S)$ induced by $\sigma$
\end{prop}
\goodbreak

\begin{proof}
Since we have $t_2(Y)=t_2(S)^\sigma$ by Lemma \ref{lem:t2^G}, 
the projection onto $t_2(Y)$ is given by the idempotent endomorphism
$e=(\bar\sigma+1)/2$ of $t_2(S)$.
Since $t_2(S)\cong t_2(Y)$ is equivalent to $e=1$, and $t_2(Y)=0$ 
is equivalent to $e=0$, the result follows.
\end{proof}

\begin{subrem}
If $S$ is a smooth minimal surface of general type with $p_g(S)=0$
and an involution $\sigma$, such that the minimal model $W$ of $Y$ is 
either an Enriques surface, a rational surface or a surface of 
Kodaira dimension equal to 1 then by \cite{GP} we have $t_2(Y)=0$.
\end{subrem}

\begin{defn}\label{def:bidouble}
A {\it bidouble cover} $f: V \to X$ between smooth projective surfaces 
is a finite flat Galois morphism with Galois group $\Z/2\times\Z/2$. 
By \cite{Pa91}, in order to define $f$ it is enough to give 
smooth divisors $D_1, D_2,D_3 $ in $X$ 
with pairwise transverse intersections and no common intersections, 
and line bundles $L_1, L_2, L_3$ such that $2L_i \equiv D_j+D_k$ 
for each permutation $(i,j,k)$ of $(1,2,3)$.
We will frequently use the fact that every nontrivial element
$\sigma$ of $G$ is an involution on $S$.
\end{defn}  

\begin{thm}\label{3-involutions}
Let $S$ be surface of general type with $p_g(S)=0$ 
which is the smooth minimal model of a bidouble cover of a surface. 
Let $Y_i$ denote the desingularization of $S/\sigma_i$, where
$\sigma_1,\sigma_2,\sigma_3$ are the non-trivial involutions
of $S$ associated to the bidouble cover.
If $t_2(Y_i)=0$ for $i=1,2,3$ then $t_2(S)=0$.
\end{thm}

\begin{proof} 
By Lemma \ref{t2=0.BC}, $t_2(S)^{\sigma}=t_2(Y_i)=0$. 
Thus each $\sigma_i$ acts as multiplication by $-1$ on $t_2(S)$. 
But since $\sigma_1\sigma_2=\sigma_3$,
$\sigma_3$ must act as multiplication by $(-1)^2=+1$.
\end{proof}

\bigskip
\section{Composed involutions}

If $S$ is a smooth minimal surface of general type, then the
linear system $|2K_S|$ determines the bicanonical map 
$\Phi_2:S\to\P^N$, where $N=\dim(2K_S)$. We say that $\Phi_2$ is 
{\it composed} with an involution $\sigma$ if
$\Phi_2$ factors through the map $S\to S/\sigma$.

Recall that $k$ denotes the number of isolated fixed points, and that 
$Y$ is a resolution of $S/\sigma$, as in \eqref{square:S-sigma}.
 
\begin{thm}\label{sigma-composed}
Let $S$ be a smooth minimal surface of general type with $p_g(S)=0$,
supporting an involution $\sigma$. Then
 \\
(1) $\Phi_2$ is composed with $\sigma$ iff $k=K_S^2+4$.

\goodbreak\noindent
If the bicanonical map $\Phi_2$ is composed with $\sigma$ then 
\\
(2) $S/\sigma$ is either rational or it is birational to an Enriques surface.
\\
(3) $-4\le K_Y^2\le0$, and $K_Y^2=0$ iff $K_Y$ is numerically zero.
\end{thm}
 
\begin{proof} 
This is proven in \cite[3.6(iv) and 3.7(iv,v)]{CCM}.
\end{proof}

\begin{cor} 
Let $S$ be a minimal surface of general type with $p_g(S)=0$ and 
let $X$ be the image of the bicanonical map $\Phi_2$. If $\Phi_2:S\to X$
is a bidouble cover, then $S$ satisfies Bloch's conjecture.
\end{cor}

\begin{proof} 
The assumption that the image $X$ of the bicanonical map of $S$ 
is a surface and that $\Phi_2$ is a morphism imply that $K^2_S\ge2$ 
(see \cite[2.1]{MP02}) and that the degree of $\Phi_2$ is $\le8$. 
If $\Phi_2$ is a bidouble cover then all three nontrivial involutions
$\sigma_i$ of $S$ (see \ref{def:bidouble}) are composed with the bicanonical 
map. By Theorem \ref{sigma-composed}, $t_2(S/\sigma_i)=0$ so
Theorem \ref{3-involutions} implies that $t_2(S)=0$. 
That is, $S$ satisfies Bloch's conjecture.
\end{proof}

Recall that $\rho(S)$ denotes the rank of the N\'eron-Severi group $NS(S)$.
 
\begin{cor}\label{D^2 bound}
Let $S$ be a minimal complex surface of general type with $p_g(S)=0$
and an involution $\sigma$. Let $D$ be the 1-dimensional part of the 
fixed locus $S^{\sigma}$. Then
 \\
(1) $\sigma$ acts as the identity on $H^2(S,\Q)$ iff $D^2=K^2_S-8$.
\\
(2) If the bicanonical map is composed with $\sigma$ and $X,Y$ are 
as in diagram \eqref{square:S-sigma}, then $D^2=K^2_S+2K^2_Y$ and
\[
K^2_S-8 \le D^2 \le K^2_S.
\]
When $D^2=K^2_S$, $S/\sigma$ is birational to an Enriques surface.
\end{cor}

\begin{proof} 
Let $t$ denote the trace of $\sigma$ on $H^2(S,\Q)\cong NS(S)_{\Q}$; 
it is at most $\rho(S)= \dim H^2(S,\Q)$. By \cite[4.2]{DMP} $t=2-D^2$.
Since $\rho(S)=10-K^2_S$ by \ref{rho+K2}, we deduce that $D^2\ge K^2_S-8$,
with equality iff $t=\rho(S)$, i.e., iff $\sigma$ acts as the identity 
on $H^2(S,\Q)$. 

Assume now that $\sigma$ is composed with the bicanonical map.
By Theorem \ref{sigma-composed}, $\sigma$ has $k=K^2_S+4$
fixed points. By \cite[3.3(ii) and 3.6(iii)]{CCM}, 
the line bundle $L$ introduced after \eqref{square:S-sigma} satisfies 
$L^2+2=-K_Y\cdot L=+K_Y^2$.
We also have $2 L^2=D^2-k$, by Lemma \ref{2L2=D2-k}.
Therefore
\[
D^2= K^2_S+ 2K^2_Y.
\]
Since $-4\le K^2_Y\le 0$ by Theorem \ref{sigma-composed}, 
we get $K^2_S-8 \le D^2 \le K^2_S$.
Theorem \ref{sigma-composed} also implies that if $K^2_Y=0$ then 
(since $Y$ is not rational) the minimal model of $Y$ 
is an Enriques surface, i.e., $2K_Y=0$. 
\end{proof}

\begin{subrem}\label{rem:rhoY}
From Lemma \ref{rho+K2}, $\rho(X)=\rho(S)+k$ equals $10-K^2_S+k$.
If the bicanonical map is composed with $\sigma$ then $k=K^2_S+4$ 
by \ref{sigma-composed}, so we get $\rho(X)=14$. Since $Y=X/\sigma$, 
we have $p_g(Y)=q(Y)=0$, $NS(Y)_{\Q}=NS(X)_{\Q}^\sigma$ and hence
$\rho(Y)\le\rho(X)$. Since $K_Y^2\le0$ and
$\rho(Y)=10-K_Y^2$ we get the bounds
\[ 10 \le \rho(Y) \le 14. \]
The equality $\rho(Y)=14=\rho(X)$ corresponds to the case $K^2_Y=-4$,
in which case $D^2=K^2_S-8$. 
\end{subrem}

\begin{ex}[$K^2_S=8$] 
Let $S$ be a minimal surface of general type with $p_g(S)=0$ and 
$K^2_S=8$, having an involution $\sigma$. 
Then $\rho(S)=10-K^2_S=2$ (Lemma \ref{rho+K2}), and 
the Hodge Index Theorem implies that
the involution $\sigma$ acts as the identity on $H^2(S,\Q)$.
By Corollary \ref{D^2 bound}, $D^2=K^2_S-8=0$.

The following facts are established in \cite[4.4]{DMP}. 
The number $k$ of the fixed points is one of 4,6,8,10,12. 
The bicanonical map is composed with $\sigma$ iff $k=12$
and in that case $Y=X/\sigma$ is rational (see \ref{sigma-composed}).
If $k=10$ then $Y$ is a rational surface and $\rho(Y)=\rho(X)=12$.
If $k=8$ the Kodaira dimension of $Y$ is 1. Finally, if $k=4,6$ the
quotient surface $Y$ is of general type.
\end{ex} 

\medskip 

Carlos Rito has analyzed the situation where $\Phi_2$ is 
{\it not} composed with $\sigma$. We extract the following result from
\cite[Thm.\,2]{Ri2}. Given a surface $S$ with involution, we construct
a resolution $Y\to S/\sigma$ and $X\to Y$ with smooth branch locus $B$ 
as in \eqref{square:S-sigma}; by blowing down curves on $Y$, we obtain a 
minimal resolution $W$ of $S/\sigma$. Let $\bar B$ denote the image
of $B$ under the proper map $Y\to W$.

\begin{thm}\label{not-composed} [Rito]
Let $S$ be a smooth minimal surface of general type with $p_g(S)=0$,
and $\sigma$ an involution on $S$ with $k$ isolated fixed points.
Suppose that the bicanonical map is not composed with $\sigma$. 
Let $W$ be a minimal model of the resolution $Y$ of $S/\sigma$ and 
let $\bar B$ be as above. If $W$ has Kodaira dimension 2, then
one of the following holds:
\begin{enumerate}
\item $K_S^2=2K_W^2$, and $\bar B$ is a disjoint union of 4 nodal curves;
\item $4\le K_S^2\le8$, $K_W^2=1$, $\bar B^2=-12$, $K_W\cdot\bar B=2$, 
\\ $k=K_S^2$ and $\bar B$ has at most one double point;
\item $6\le K_S^2\le8$, $K_W^2=2$, $\bar B^2=-12$, $K_W\cdot\bar B=2$, 
\\ $k=K_S^2-2$ and $\bar B$ is smooth.
\end{enumerate}
\end{thm}

\begin{cor}\label{S/s if K2=3}
Let $S$ be a smooth minimal surface of general type with $p_g(S)=0$,
having a nontrivial involution $\sigma$. If $K_S^2=3$ then the 
resolution $Y$ of $S/\sigma$ is not of general type. 
In particular, $S/\sigma$ satisfies Bloch's conjecture.
\end{cor}

\begin{proof}
If the bicanonical map of $S$ is composed with $\sigma$, then
Theorem \ref{sigma-composed}(2) states that 
$Y$ is either rational or its minimal model is an Enriques surface.
If the bicanonical map of $S$ is {\it not} composed with $\sigma$,
and $Y$ is of general type (has Kodaira dimension 2),
then Theorem \ref{not-composed} shows that $K_S^2$ cannot be 3.
\end{proof}

\begin{ex} 
Here are some examples in which the minimal model $W$ of $S/\sigma$
is of general type, and the divisorial part $D$ of $S^\sigma$ satisfies $D=0$.
Let $k$ be the number of isolated fixed points of $\sigma$. 
\\
1) The Barlow surface $W$ in \cite{Bar} is of general  
type with $K^2_W=1$ and $p_g(W)=0$. It is the minimal model
of $S/\sigma$, where $S$ is a Catanese surface with $K^2_S=2$, 
by factoring with an involution $\sigma$ which has only 4 isolated 
fixed points. Hence $D=0$ so that $\bar B= \sum_{1\le i\le k}C_i$ 
is the disjoint union of 4 nodal curves. Therefore $W$ 
satisfies condition (1) of Theorem \ref{not-composed}. 
In \cite{Bar} it is proved that $t_2(W)=0$.
\\
2) Another example comes from \cite[Ex.\,4.3]{MP01a};
cf.\ \cite[4.4(i)]{DMP}. 
Using Pardini's method for the group $G=(\Z/2)^4$ 
(as in Lemma \ref{Pardini's cover} below),
one first constructs two smooth $G$-covers $C_1$, $C_2$ of $\P^1$ 
of genus 5 as in \cite [Ex.\,4.3]{MP01a}. The authors choose a subgroup
$\Gamma$ of $G\times G$ with $|\Gamma|=16$ acting freely
on $C_1\times C_2$; the quotient $S=(C_1\times C_2)/\Gamma$ 
is the desired minimal surface of general type with 
$K^2_S =8$ and $p_g(S)=q(S)=0$.

For this $S$, it is shown in {\it loc.\,cit.}\ that the bicanonical
map is birational, hence not composed with any involution.
The construction of $S$ starts with a basis 
of $G$ and constructs the $G$-covers $C_i\to\P^1$ so that any of the
11 other nontrivial elements of $G$ acts freely on both curves, 
and acts on $S$ (via the diagonal action) with $D=0$ and 
$k=K_S\cdot D+4=4$ fixed points. 

Since $D=0$ and $k=4$, a minimal model $W$ of $S/\sigma$ has $K^2_W=4$
by \cite[4.4(i)]{DMP}.
Thus $W$ satisfies condition (1) of Theorem \ref{not-composed}.
In fact, the branch locus of $X\to Y$ is the disjoint union of 
the 4 lines over the fixed points of $\sigma$, so
$\bar B$ is the disjoint union of the 4 nodal curves $C_i$. 

Since the motives $h(C_i)$ are finite dimensional, so are 
$h(C_1\times C_2)$ and $h(S)=h(C_1\times_G C_2)$. 
Hence $t_2(S)$ is also finite dimensional. As $p_g(S)=0$,
Lemma \ref{t2=0.BC} implies that $S$ satisfies Bloch's conjecture.
\end{ex}

 
\goodbreak 
\section{Numerical Godeaux surfaces ($K^2=1$)}\label{sec:K2=1}

In the next few sections, we will give examples of complex surfaces
$S$ of general type, having an involution $\sigma$, with $p_g(S)=0$
and $3\le K^2_S\le7$, for which Bloch's conjecture holds. In this
section and the next we deal with the cases $K^2_S=1,2$.

Complex surfaces of general type with $p_g=q=0$ and $K^2_S=1$ are called 
{\it numerical Godeaux surfaces}. Examples of Godeaux surfaces are obtained
as quotients of a Catanese surface under the action of a finite
group. (A Catanese surface is a minimal surface $V$ of general type
with $p_g(V)=0$ and $ K^2_V=2$).

Barlow gave two such examples of numerical Godeaux surfaces $S$ 
in the 1980's, both with with $\pi_1(S)=\Z/5$ and $\pi_1(S)=\{1\}$, 
and showed in \cite{Bar} that Bloch's conjecture holds for these $S$.

By a {\it typical Godeaux surface} we mean the quotient surface of 
the Fermat quintic in $\P^3$ by a cyclic group of order 5. 
It has $\rho(S)=9$ and $K_S^2=1$.  By
\cite{GP1}, $S$ has a finite dimensional motive, so it satisfies
Bloch's conjecture. This case was also handled in \cite{IM}.

We now turn to numerical Godeaux surfaces with an involution $\sigma$.
These were classified in \cite{CCM}; if $S/\sigma$ is rational 
they are either of {\it Campedelli type} or {\it Du Val type}; 
if $S/\sigma$ is an Enriques surface they are of {\it Enriques type}.
This exhausts all cases, since (by \cite[4.5]{CCM}) the
bicanonical map is composed with $\sigma$ and $\sigma$ has exactly
$k=5$ fixed points; as noted in Theorem \ref{sigma-composed}, this
implies that $S/\sigma$ is either rational or birational to an
Enriques surface.

If $S$ is of Du Val type, there is an \'etale double cover $\tilde S$ 
of $S$ with $p_g(\tilde S)=1$ and $K^2=2$; see \cite[8.1]{CCM}. 
Since $\pi_1(\tilde S)$ is the kernel of $\pi_1(S)\to\Z/2$, 
we also have $q(\tilde S)=0$. In addition, $\Pic(\tilde S)=NS(\tilde S)$ 
is either torsion free or has a torsion subgroup of order 2.

\begin{thm}
Let $S$ be a numerical Godeaux surface of Du Val type such that 
the torsion subgroup of $\Pic(S)$ has order 2. Let $\tilde S \to S$ 
be the \'etale double cover of $S$ associated to the non trivial
2-torsion element in $\Pic(S)$. If $\tilde S$ has bicanonical map of
degree 4, then $S$ satisfies Bloch's conjecture.
\end{thm}

\begin{proof} 
By \cite[\S3]{CD}, $\Pic(\tilde S)$ is torsion free and the image of the
bicanonical map $\Phi:\tilde S\to\P^3$ is a smooth quadric. By
Lemma \ref{lem:pg=1} below, $\tilde S$ has a finite dimensional
motive; this implies that $h(S)$ is also finite dimensional. As $p_g(S)=0$, 
Lemma \ref{t2=0.BC} implies that $S$ satisfies Bloch's conjecture.
\end{proof}

Lemma \ref{lem:pg=1} is based upon the Catanese--Debarre paper \cite{CD}.

\begin{lem}\label{lem:pg=1} 
Let $S$ be a surface of general type with $p_g(S)=1$, $q(S)=0$ and
$K^2_S=2$. Suppose that the bicanonical map $\Phi:S\to\P^3$ has
degree 4 and is a morphism onto a smooth quadric $\Sigma$. 
Then the motive of $S$ is finite dimensional.
\end{lem}

\begin{proof} 
As shown in \cite[3.2]{CD}, the bicanonical map 
$S\to\Sigma\subset\P^3$ is a bidouble cover with Galois group
$G=\{1,\sigma_1,\sigma_2,\sigma_3\}$. Set $S_i=S/\sigma_i$, with
$i=1,2,3$. By \cite[3.1]{CD}, both $S_1$ and $S_2$ are rational
surfaces, while the minimal smooth surface $M$ over $S_3$ is a K3
surface. Moreover $S_3$ has 10 nodal points, and their inverse images
in $M$ are 10 disjoint smooth rational curves $F_1,\cdots,F_{10}$.

Since $t_2(S)^{\sigma_i}= t_2(S_i)$ and $t_2(S_1)=t_2(S_2)=0$, it
follows that $\sigma_1$ and $\sigma_2$ act as $-1$ to $t_2(S)$.
Hence $\sigma_3=\sigma_1\circ\sigma_2$ acts as $+1$. 
It follows that $t_2(S)=t_2(S_3)=t_2(M)$.

On $S_3$, the involutions $\bar\sigma_1$ and $\bar\sigma_2$ agree
and induce an involution $\sigma$ on $M$. Since $S_3/\bar\sigma_1=\Sigma$,
$M^{\sigma}$ is the union of the ten curves $F_i$.
Because $\Sigma$ is a rational surface, 
$H^{2,0}(M)^{\sigma}=H^{2,0}(\Sigma)=0$, 
so $\sigma$ acts as -1 on $H^{2,0}(M)\cong\C$.
It follows from \cite[3.1]{Zh} that the Neron-Severi group of
$M$ has rank $\rho(M)=20$. By Theorems 2 and 3 of \cite{Ped}, 
this implies that the motive $h(M)$ is finite dimensional in $\sM_{rat}$.
In particular $t_2(M)=t_2(S)$ is finite dimensional. Since $h_0(S),
h_4(S)$ and $h^{alg}_2(S)$ are finite dimensional, it follows that
$h(S)$ is finite dimensional.  
\end{proof}

\section{Numerical Campedelli surfaces ($K^2=2$)}\label{sec:K2=2}

Complex surfaces of general type with $p_g=q=0$ and $K^2_S=2$ 
are called {\it numerical Campedelli surfaces}.

Now let $S$ be a numerical Campedelli surface with an involution
$\sigma$; these have been classified in \cite{CMP}. 
By \cite[3.32(i,iv)]{CCM}, $\sigma$ has either $k=4$ or $k=6$ isolated 
fixed points. By Theorem \ref{sigma-composed}, the bicanonical map 
$\Phi:S\to\P^2$ is composed with $\sigma$ iff $k=6$; in this case 
Corollary \ref{D^2 bound} yields $-6 \le D^2 \le 2$. 
(This result appeared in \cite[3.1(ii)]{CMP}).

Campedelli surfaces with fundamental group $(\Z/2)^3$ satisfy Bloch's
conjecture, by the results in \cite{IM}. In \cite[2.4]{Vois}, 
Voisin proves that Bloch's conjecture holds for a family of Campedelli
surfaces constructed from $5\times5$ symmetric matrices $M(a)$, with 
$a\in \P^{11}$, of linear forms on $\P^3$ satisfying certain conditions.

Numerical Campedelli surfaces which arise as a $(\Z/2)^3$-cover of
$\P^2$ branched along 7 lines are called {\it Classical Campedelli
surfaces}. They were constructed by Campedelli, and later by
Kulikov in \cite{Ku}.  All 7 of the nontrivial involutions of
classical Campedelli surfaces are composed with bicanonical map.

\begin{thm}
Classical Campedelli surfaces satisfy Bloch's conjecture.
\end{thm}

\begin{proof} 
The automorphism group of $S$ coincides with $(\Z/2)^3$; see
\cite[Thm.\,4.2]{Ku}. By \cite[5.1]{CMP}, the bicanonical map
is composed with every nontrivial involution $\sigma$ of $S$ and 
each quotient $S/\sigma$ is either rational or else it is 
birationally equivalent to an Enriques surface. 
By Theorem \ref{3-involutions} 
we get $t_2(S)=0$, and hence $S$ satisfies Bloch's conjecture.
\end{proof}

Here is one such family of classical Campedelli surfaces. 
We fix 7 distinct lines $L_i$ in $\P^2$ such that at most 3
pass through the same point, and enumerate the nontrivial elements of
$G=(\Z/2)^3$ as $g_1,\dots,g_7$, so that
$g_1,g_2,g_3$ generate $G$ and if $g_i+g_j+g_k=0$ then
$L_i$, $L_j$ and $L_k$ do not pass through the same point.
Fix characters $\chi_1,\chi_2,\chi_3$ generating $G^*$.
By \cite{Pa91}, the equations 
$2\mathcal{L}_i=\sum_{j=1}^7\epsilon_{ij}L_j$ 
($i=1,2,3$) determine a normal $G$-cover $V$ of $\P^2$ as long as
$\epsilon_{ij}$ is $1$ if $\chi_i(g_j)=-1$ and zero otherwise.
(See Example 1 of \cite[\S5]{CMP}, or \cite[4.1]{Ku}.) 
The surface $S$ is obtained by resolving the singular points of $V$, 
which only lie over the triple intersection points of the lines
$L_i$ in the plane.

\section{The case $K_S^2=3,4$}\label{sec:K2=3,4}

Let $S$ be a minimal surface of general type with $p_g(S)=0$
and $K^2_S=3$. For each nontrivial involution $\sigma$ of $S$,
Corollary \ref{S/s if K2=3} and Lemma \ref{t2=0.BC} 
imply that $t_2(S/\sigma)=0$.  Therefore if $S$ is birational to
a bidouble cover, then $S$ satisfies Bloch's conjecture
by Theorem \ref{3-involutions}. 

\begin{ex}[$\mathbf{K^2_S=3}$] 
Rito gives an example in \cite[5.2]{Ri1} in which $S$ is the 
minimal model of a bidouble cover of $\P^2$ and $K^2_S=3$.
Write $\sigma_1,\sigma_2,\sigma_3$ for the 3 involutions of $S$ 
corresponding to the bidouble cover. Rito shows that
the bicanonical map is not composed with $\sigma_1$ or $\sigma_2$ but is
composed with $\sigma_3$.  If $W_i$ is a minimal model of the
desingularization of $S/\sigma_i$, he shows that $W_1$ is an
Enriques surface, $W_2$ has Kodaira dimension 1 and $W_3$ is a
rational surface.  Therefore $t_2(W_i)=0 $ for $i=1,2,3$. 
As remarked above, this implies that
$S$ satisfies Bloch's conjecture. 

Similar examples with $4 \le K^2_S \le 7$ have
been constructed in \cite{Ri2}.
\end{ex}

\medskip 
Y.\,Neum and D.\,Naie \cite{Na} constructed a family of surfaces of 
general type with $K^2=4$ as double covers of an Enriques surface with eight
nodes.  I.\,Bauer and F.\,Catanese proved in \cite{BC2}
that the connected component of the moduli space corresponding to
this family is irreducible, normal, unirational of dimension 6. In
Remark 3.3 of {\it op.\ cit.} they noticed that the proof, given in an
unpublished manuscript by Keum, of Bloch's conjecture for a subfamily
of dimension 4 of the connected component of the moduli space of all
Keum-Naie surfaces, can be extended to the whole family.

Using the construction of Keum-Naie surfaces given in \cite{BC2},
we will now give a motivic proof of Bloch's conjecture for these surfaces.

\begin{thm} 
Let $S$ be a Keum-Naie surface with $K^2_S=4$ and $p_g(S)=0$. 
Then $S$ satisfies Bloch's conjecture.
\end{thm}

\begin{proof} 
Let $E_1=\C/\Lambda_1$ and $E_2=\C/\Lambda_2$ with 
$\Lambda_i=\Z e_i \oplus \Z e'_i$ be two elliptic curves and 
let $\Gamma$ be the group of affine transformations generated by 
$\gamma_1$, $\gamma_2$ and by the translations $(e_1,e'_1,e_2, e'_2)$, 
where $\gamma_1$ and $\gamma_2$ are given by
$$
\gamma_1(z_1,z_2)= (z_1+e_1/2,-z_2) \ ; \ 
\gamma_2(z_1,z_2)= (-z_1,z_2 +e_2/2).
$$
Then $\gamma^2_i=e_i$, for $i=1,2$ and 
$\Gamma=\langle\gamma_1,e'_1,\gamma_2,e'_2\rangle$. 
Let $G=\Z/2\times\Z/2$; then we have an exact sequence of groups, 
see \cite[3.1]{BC2}
$$
1 \to \Z^4 \simeq\langle e_1,e'_1,e_2,e'_2\rangle \to \Gamma \to G \to 1.
$$
Here $\gamma_1 \mapsto(1,0)=g_1$, $\gamma_2 \mapsto (0,1)=g_2$ and 
$(\gamma_1 \circ \gamma_2 \mapsto (1,1)= g_1 \circ g_2$.

The action of $\Gamma$ on $E_1\times E_2$ induces a faithful action of the 
group $G$ on $E_1\times E_2$. $\gamma_1, \gamma_2$ have 
no fixed points on $E_1 \times E_2$, while the involution 
$\gamma_1 \circ \gamma_2$ has 16 fixed points. The quotient 
$Y=E_1\times E_2/G$ is an Enriques surface having 8 nodes. 
In \cite[1.4]{BC2} they lift the $G$-action on $E_1\times E_2$ 
to a double cover $\tilde X \map{\pi} E_1\times E_2$, branched in a 
$G$-invariant divisor, in such a way that $G$ acts freely on $\tilde X$.
Let $f:\tilde S \to \tilde X$ be a minimal resolution of singularities 
of $\tilde X$ and 
set $S=\tilde S/G$. Then $S$ is a Keum-Naie surface with $K^2_S=4$ 
and $p_g(S)=q(S)=0$. 
The bicanonical map $\Phi_2:S\to\P^4$ of $S$ has degree 4 and
the composition of $h:\tilde S \to S$ with $\Phi_2 $ factors through 
the product $E_1 \times E_2$. Moreover these maps are invariant
under the action of the group $G$ as well as under the action of the 
automorphism
\[
\gamma_3(z_1,z_2)= (-z_1 +e_1/2,z_2).
\]
Therefore the composition $\Phi_2\circ h$ factors through the 
$\Z/2$-quotient $\Sigma$ of the Enriques surface $Y$ by the action 
of the involution $\gamma_3$. $\Sigma \subset \P^4$ is a Del Pezzo surface
 of degree 4. In summary we have a commutative diagram
$$\CD
 \tilde S@> {h}>> S=\tilde S/G@>{\Phi_2}>> \Sigma   \\
 @V{\pi \circ f}VV  @V{p}VV  @V{=}VV \\
E_1 \times E_2 @>{g}>> Y=(E_1 \times E_2)/G@>{s}>> \Sigma \\
\endCD$$
where the map $h$ is an \'etale $G$-covering, 
the map $\pi\circ f$ of degree 2 is the Albanese map of $\tilde S$,
and where $p$ and $s$ have degree 2. 
Therefore the Keum-Naie surface $S$ is a 
$H$-covering of the rational surface $\Sigma$, with $H\simeq\Z/2\times\Z/2$
and $\tilde S$ is an $H\times G$-covering of $\Sigma$. Let us enumerate:
$$
H=\{1,\sigma_1, \sigma_2, \sigma_3\}.
$$ 
The 3 involutions $\sigma_i$ are all composed with $\Phi_2$.
Hence $t_2(S/\sigma_i)=0$ for $i=1,2,3$ by Theorem \ref{sigma-composed}.
From Theorem \ref{3-involutions} we get $t_2(S)=0$.
\end{proof}
 
\section{Fibrations on surfaces}\label{sec:fibrations}

A surface $S$ is said to be a {\it product-quotient surface} 
if it is birational to a quotient $(C_1\times C_2)/G$ of the product 
$C_1 \times C_2$ of two curves of genera $\ge2$ 
by the action of a finite group $G$. Since the motives $h(C_i)$ are 
finite dimensional, so are $h(C_1\times C_2)$ and $h(S)=h(C_1\times C_2)^G$. 
Since $t_2$ is a birational invariant, $t_2(S)$ is also finite dimensional.
If $p_g(S)=0$, this implies that $S$ satisfies Bloch's conjecture
(by Lemma \ref{t2=0.BC}).

A complete classification is given in \cite{BCGP} of surfaces $S$ with
$p_g(S)=q(S)=0$, whose canonical models arise as product-quotient
surfaces. If $G$ acts freely then the quotient surface is minimal of
general type and it is said to be {\it isogenous to a product}. If $S$
is isogenous to a product then $K^2_S=8$, 
see \cite[0.1]{BCG} or \cite[Thm.\,4.3]{BCP}.

In the case when $G$ acts freely on both $C_1$ and $C_2$, then the
projection $ C_1 \times_G C_2 \to C_2/G$ has fibers $C_1$. More
generally, a fibration $S \to B$ from a a smooth projective surface
onto a smooth curve is said to be {\it isotrivial} if the smooth
fibers are mutually isomorphic.

\begin{thm}\label{thm:isotriv=fd}
Let $S$ be a complex surface of general type with $p_g=0$ which
has a fibration $S\map{\pi}B$ with $B$ a smooth curve of genus $b$, 
and general fibre a curve $F$ of genus $g\ge1$. 
If $\pi$ is isotrivial, then $t_2(S)=0$, and $S$ satisfies Bloch's conjecture.
\end{thm}

\begin{proof} 
By \cite[2.0.1]{Se} there is a finite group $G$ acting on the fiber $F$
and a Galois cover $C\to B$ so that $B=C/G$ and
$S$ is birational to $S'=F\times_G C$.
\end{proof}

One source of isotrivial fibrations comes from the following
observation of Beauville. Let $S$ be a smooth projective complex
surface and $S\to B$ a fibration with general fibre a smooth curve $F$
of genus $g\ge1$.  Let $b$ denote the genus of the curve $B$.
Beauville proved in \cite{Be} that
\begin{equation}\label{eq:Beauville}
K^2_S \ge 8(b-1)(g-1),
\end{equation}
and if equality holds then the fibration is isotrivial.

Let $S$ be a smooth complex surface with a fibration $f:S\to\P^1$ with
general fiber $F$ a smooth curve of genus $g(F)$. 
In many cases, we can contruct a finite map $h:C\to\P^1$
with $C$ smooth such that the normalization $X$ of $C\times_{\P^1}S$
is nonsingular and the map $\tilde h:X\to S$ is \'etale
Because $\tilde h$ is \'etale, $X$ is smooth and $K_X=\tilde h^*K_S$, 
we have $K^2_X=\deg(\tilde h) K^2_S$. This information is summarized
in the commutative diagram:
\begin{equation}\label{square:XSC}
\CD
  X @>{\tilde h}>>S \\
 @VV{f_C}V  @VV{f}V  \\
  C @>{h}>> \P^1.
\endCD
\end{equation}

\begin{lem}\label{4.3} 
Suppose that the singular fibers of a fibration $S\to\P^1$ consist only in 
double fibers, over points $P_1,\dots,P_r$ of $\P^1$, and that there is
a smooth cover $h:C\to\P^1$ branched along the $P_k$,
with $h^{-1}(P_k)$ consisting only in double points.
Then, in the diagram \eqref{square:XSC}, $X$ is nonsingular and 
fibered over $C$, and $\tilde h$ is an \'etale map.
\end{lem}

\begin{proof} 
The base change, $C\times_{\P^1}S\to S$ is \'etale except over
the double fibers, where it is a simple normal crossings divisor.
Therefore the normalization $X$ of $C\times_{\P^1}S$ is \'etale over $S$.
\end{proof}

\begin{ex}\label{doublecover} 
Suppose that $f:S\to\P^1$ has an even number of singular fibers,
all double fibers. Let $C$ be the double cover of $\P^1$ branched at
the points of $\P^1$ supporting the double fibers of $f$. Then the
hypotheses of Lemma \ref{4.3} are satisfied, so we have a diagram 
\eqref{square:XSC} with $\tilde h$ \'etale.
\end{ex} 

\paragraph{\bf Pardini's method}
We recall Pardini's method (from \cite[(2.21)]{Pa91}) for producing a 
smooth cover $C\to\P^1$, branched along a given set of 
$r$ points $P_k$, together with a faithful action of the group
$G=(\Z/2)^s$ on $C$ so that $C/G=\P^1$.
Fix linearly independent 1-dimensional characters $\chi_1,\dots,\chi_s$ 
of $G=(\Z/2)^s$ and enumerate the $2^s-1$ cyclic subgroups
$H_j$ of $G$ in some order. Define integers $\epsilon_{ij}$ to be 
1 if the character $\chi_i$ is nontrivial on $H_j$ and 0 otherwise.

Suppose that we can partition the $P_k$ into $2^s-1$ subsets $D_j$
of cardinality $n_j$, and find integers $L_1,\dots,L_s$
such that $2L_i=\sum\epsilon_{ij}n_j$ for all $i$. 
Regarding each $D_j$ as an effective divisor of degree $n_j$ and
the $L_i$ as the degrees of line bundles, 
this yields a family of ``reduced building data'' in the sense of
\cite[Prop.\,2.1]{Pa91}. Pardini constructs a $G$-cover of $\P^1$
from this data in {\it loc.\,cit.}
Since the fiber over each $P_k$ has $2^{s-1}$ double points,
the ramification divisor on $C$ has degree $2^{s-1}r$ and
 $C$ has genus $2^{s-2}r+1-2^s$, by Hurwicz' theorem.
The case $s=1$ recovers the classical result that any even number
of points in $\P^1$ forms the branch locus of a double cover.
Here is the case $s=2$:

\begin{lem}\label{Pardini's cover}
Let $P_1,\cdots P_r$ be distinct points on $\P^1$, $r\ge3$. 
Then there exists a smooth curve $C$ and a $(\Z/2)^2$-cover 
$h:C \to \P^1$ such that the ramification divisor on $C$ has 
degree $2r$ and $C$ has genus $r-3$.
\end{lem}

\begin{proof}
We use Pardini's method with $s=2$. We need to partition the points
into three subsets of cardinalities $n_i$, corresponding to the
three subgroups $H_1,H_2,H_3$ of order~2 in $(\Z/2)^2$. 
As above, we need to solve the equations $n_1+n_2+n_3=r$, $n_2+n_3=2L_1$ 
and $n_1+n_3=2L_2$ for positive integers $n_i$ and $L_j$.
If $r$ is even and at least 4, we can take $n_1=n_2=2$ and $n_3=r-4$; 
if $r$ is odd and at least 3, we can take $n_1=n_3=1$ and $n_2=r-2$.
\end{proof}

\begin{thm}\label{K^2 bound} 
Let $S$ be a minimal surface of general type with $p_g(S)=0$. 
Suppose that there exists a fibration $f: S \to \P^1$ with 
general fiber $F$ and $r$ double fibers as singular fibers. Then
$$
K^2_S \ge 2(r-4) (g(F)-1).
$$
If equality holds then $S$ satisfies Bloch's conjecture.
\end{thm}

\begin{proof} 
By Lemma \ref{Pardini's cover} we can find a $(\Z/2)^s$-cover 
$h:C\to \P^1$ branched over the $r$ points of $\P^1$ corresponding to 
the double fibres of $f$. We have seen that the degree of the 
ramification divisor of $h$ on $C$ is $2^{s-1}r$, and that $C$ has
genus $2^{s-2}r+1-2^s$.
By Lemma \ref{4.3} we can find a square \eqref{square:XSC} with
$\tilde h:X\to S$ \'etale. By Beauville's bound \eqref{eq:Beauville},
\[
K^2_X \ge 8(2^{s-2}r-2^s)(g(F)-1),
\]
and if equality holds then the fibration $X \to C$ is isotrivial.
Since $K^2_X=2^s K^2_S$, we have
 \[
K^2_S= 2^{-s} K^2_X \ge 2(r-4)(g(F)-1).
 \]
and if equality holds then (by Theorem \ref{thm:isotriv=fd})
$t_2(S)=0$ and $S$ satisfies Bloch's conjecture.
\end{proof}

\begin{cor}\label{K2=8rational}
Let $S$ be a minimal complex surface of general type with $K^2_S=8$.
If $S$ has an involution $\sigma$ such that $S/\sigma$ is rational,
then $S$ satisfies Bloch's conjecture.
\end{cor}

\begin{proof} 
By \cite[Thm.\,2.2]{Pa03}, $S$ has a fibration $S\to\P^1$
with general fiber $F$ and $r$ double fibers, such that either 
$r=6$ and $g(F)=3$ or $r=5$ and $g(F)=5$. In both case there is an equality 
$K^2_S=2(r-4)(g(F)-1)$. 
By Theorem \ref{K^2 bound}, $S$ satisfies Bloch's conjecture.
\end{proof}

\section{Inoue's surface with $K^2=7$}\label{sec:K2=7}

Until recently, the only known family of examples of surfaces $S$ of
general type with $K^2_S=7$ and $p_g(S)=0$ was constructed by M.\,Inoue
in \cite{I}. It is a quotient of a complete intersection in the
product of four elliptic curves, by a free action of $\Z/5$. (Another
family was found recently by Y.\,Chen in \cite{Ch}; see Remark
\ref{rem:Chen}).

An alternative description of Inoue's surface as a bidouble
cover of a rational surface was given in \cite{MP01a}.  Let $\Pi$ be
the blow up of $\P^2$ in 6 points $P_1,\dots,P_6$ as in 
\cite[Ex. 4.1]{MP01a} and let $\pi:X\to\Pi$ be the bidouble cover with
 branch locus $D=D_1+D_2+D_3$, where
\begin{align*}
D_1 = \Delta_1 + F_2 + L_1 +L_2 \ ; \ &
D_2 = \Delta_2 +F_3\ ; \ 
\notag\\
D_3 = \Delta_3 + F_1 + F'_1 & + L_3 +L_4.
\end{align*}
Here $\Delta_i$ ($i=1,2,3$) is the strict transform in $\Pi$ of 
the diagonal lines $P_1P_3$, $P_2P_4$ and $P_5P_6$, respectively;
$L_i$ is the strict transform of the line between $P_i$ and $P_{i+1}$
of the quadrilateral $P_1,P_2,P_3,P_4$;
$F_1$ is the strict transform of a general conic through 
$P_2,P_4,P_5,P_6$ and $F'_1\in|F_1|$; 
$F_2$ is the strict transform of a general conic through 
$P_1,P_3,P_5,P_6$; and $F_3$ is that of a general conic through
$P_1,P_2,P_3,P_4$.

The image of the morphism $f:\Pi\to\P^3$ given by $|{-K}_{\Pi}|$ is
a cubic surface $V\subset\P^3$; $f$ contracts each $L_i$ to a
point $A_i$, and is an isomorphism on $\Pi\setminus\cup_i L_i$.
The (set-theoretic) inverse image in $X$ of the 4 lines $L_i$ 
is the disjoint union of two (-1)-curves $E_{i,1}$ and $E_{i,2}$.
The surface $S$ is obtained by contracting these eight exceptional 
curves on $X$; the results in \cite{MP01a} show that 
$S$ is a surface of general type with $p_g(S)=0$, $K^2_S=7$, that
the bicanonical map $\Phi_2: S\to \P^7$ has degree 2, and the 
bicanonical map is one of the maps $S\to S/\sigma$ associated to 
the bidouble cover.

\begin{thm}\label{K2=7} 
Inoue's surface $S$ with $K^2_S=7$ satisfies Bloch's conjecture.
\end{thm}

\begin{proof}
Let $\sigma_1$, $\sigma_2$, $\sigma_3$ be the nontrivial involutions of
$X$ over $\Pi$, or equivalently, of $S$ over $V$.  We will determine
the number $k_i$ of isolated fixed points of $\sigma_i$ on $S$ in
Lemma \ref{k-Inoue} below. Let $Y_i$ be the desingularization of $S/\sigma_i$
given by \eqref{square:S-sigma}.  For $\sigma_1$ we have
$k_1=11=K_S^2+4$ so, by Theorem \ref{sigma-composed}, 
$\Phi_2$ is composed with the involution $\sigma_1$ and 
$Y_1$ is either rational or birational to an
Enriques surface.  In particular, $t_2(Y_1)=0$.

Since $k_2$ and $k_3$ are less than $K^2_S+4$, $\Phi_2$ is 
{\it not} composed with either $\sigma_2$ or $\sigma_3$, 
by Theorem \ref{sigma-composed}. By Theorem \ref{not-composed}, 
the minimal models of $Y_2$ and $Y_3$ cannot have Kodaira dimension 2,
because $K^2_S=7$ is odd and $k_2=k_3=9=K^2_S+2$.  
It follows from \cite{BKL} that $Y_2$ and $Y_3$ satisfy 
Bloch's conjecture, and so $t_2(Y_2)=t_2(Y_3)=0$. 
By Theorem \ref{3-involutions}, 
this shows that $t_2(S)=0$ and finishes the proof.
\end{proof}

\begin{lem}\label{k-Inoue} 
The involutions $\sigma_i$ on $S$ have $k_1=11$, $k_2=9$ and $k_3=9$ 
fixed points, respectively.
\end{lem}

\begin{proof} 
There is a smooth bidouble cover $p:S\to V$, 
where $V$ is obtained from $\Pi$ by contracting the curves $L_i$ to 
4 singular points $A_1,\cdots A_4$. Hence we get a commutative diagram 
$$\CD
 X @>{\pi}>> \Pi \\
 @V{g}VV  @VV{f}V \\ 
 S @>{p}>> V \\ 
\endCD$$
where $g(E_{i,j})=Q_{i,j}$ and $p(Q_{i,j})=A_i$ for
$i=1,\cdots\!,4$ and $j=1,2$. The divisors on $\Pi$ satisfy the 
following relations, see \cite[Ex.4.1]{MP01a}:
\begin{align*}
-K_{\Pi} \equiv \Delta_1+\Delta_2+\Delta_3; \quad &
F_i \equiv \Delta_i+\Delta_{i+2}, i\in\Z/3; \\
\Delta_i\cdot L_j=0; \quad& \Delta_i\cdot F_j=2\delta_{i j}.
\end{align*}
From these relations it is easy to calculate that the number $k'_i$ of
isolated fixed points of $\sigma_i$ on $X$ is: $k'_1=(D_2 \cdot D_3)=7$;
$k'_2=(D_1\cdot D_3)=9$; and $k'_3=(D_1\cdot D_2)=5$. 
The surface $S$ may have additional isolated fixed points
$Q_{i,j}$, because some of the curves $E_{i,j}$ may be fixed on $X$.

By construction the divisorial part of the fixed locus of $\sigma_i$
on $X$ lies over $D_i$. Since $L_1+L_2$ is contained in $D_1$, the
fixed locus of $\sigma_1$ on $X$ contains the four $(-1)$-curves
$E_{i,1}$ ($i=1,2$) lying over $L_1$ and $L_2$. Their images $Q_{i,1}$
and $Q_{i,2}$ are fixed points of $\sigma_1$ on $S$, so $\sigma_1$ has
$k_1=k'_1+4=11$ isolated fixed points on $S$. Similarly, since $L_3$ and
$L_4$ are contained in $D_3$ the fixed locus of $\sigma_3$ on $S$
contains the four $(-1)$-curves $E_{3,j}$ and $E_{4,j}$, so $\sigma_3$
has $k_3=k'_3+4=9$ isolated fixed points on $S$. Finally, we have
$k_2=k'_2=9$ because none of the $L_i$ are contained in the divisorial
part $D_2$ of $\sigma_2$.
\end{proof}

\begin{rk} 
In the recent preprint \cite{LS12}, Lee and Shin consider the case of the
Inoue's surface $S$ with $K^2_S=7$ and $p_g(S)=0$ and compute
$K^2_{W_i}$, where $W_i$ is a minimal model of $S/\sigma_i$
($i=1,2,3$). In particular they show that $W_1$ and $W_3$ (hence $S/\sigma_1$
and $S/\sigma_3$) are rational, while $W_2$ (and hence $S/\sigma_2$) is
birational to an Enriques surface. Theorem \ref{K2=7}
also follows from this via Theorem \ref{3-involutions}.

In addition, Bauer's recent preprint \cite{Bau} gives a proof of 
Theorem \ref{K2=7}. Her proof is similar to ours.
\end{rk}

\begin{rk}\label{rem:Chen}
Recently Y.\ Chen in \cite{Ch} has produced a family of surfaces of
general type with $K^2=7$ and $p_g=0$, whose bicanonical map is not
composed with any involution on $S$. Chen has verified that Bloch's
conjecture holds for these surfaces.
\end{rk} 

\section{Burniat surfaces and surfaces with $K^2=6$, $K^2=7$} 

In this section we show
that Burniat surfaces satisfy Bloch's conjecture and consider other
surfaces with $K^2=6$.

\medskip
{\it Burniat surfaces}

\medskip\noindent
Burniat surfaces are certain surfaces of general type with $2\le K_S^2\le 6$.
It is proven in \cite[Thm.\,2.3]{BC1} that 
Burniat surfaces are exactly the surfaces constructed by 
Inoue in \cite{I} as the $(\Z/2)^3$-quotient of an invariant 
hypersurface in the product of 3 elliptic curves.

To construct them, one first forms a del Pezzo surface $\Pi$ as the
blowup of the plane at 3 non-collinear points $P_1,P_2,P_3$, see \cite{Pet}.
Then one forms a bidouble cover $\bar S$ of $\Pi$
whose branch locus is the union of the exceptional curves and
the proper transform of 9 other lines. Then
$S$ is the minimal resolution of $\bar S$, and $\bar S$ has 
$k=6-K_S^2$ singular points.

If $K_S^2=6$ (so $S=\bar S$), these are called {\it primary} Burniat surfaces. 
If $K_S^2=4,5$ (so $k=1,2$) $S$ is called {\it secondary}; 
if $K_S^2=2,3$ (so $k=3,4$) $S$ is called {\it tertiary}.
see \cite[2.2]{BC1}.

An important feature of the Burniat surfaces 
$S$ is that their bicanonical map is the composition of 
a bidouble cover $S\to \Pi'$ onto a normal del Pezzo surface
$\Pi'$ (a blowup of $\Pi$)
followed by the anticanonical embedding of $\Pi'$ into $\P^6$. 
(See \cite[3.1]{MP01} and \cite[4.1]{BC1}).
Thus $\Phi_2$ is composed with each of
the three nontrivial involutions $\sigma_i$ on $S$.

\begin{thm}\label{thm:Burniat} 
Every Burniat surface $S$ satisfies Bloch's conjecture
\end{thm} 

\begin{proof} 
As noted above, the bicanonical map $\Phi_2 $ is the composition of a
bidouble cover $S\to\Pi'$ with the anticanonical embedding of $\Pi'$.
Thus $\Phi_2$ is composed with each of the three nontrivial involutions 
$\sigma_i$ on $S$ associated to this cover. By 
Theorem \ref{sigma-composed} each $S/\sigma_i$ is either rational or 
birational to an Enriques surface, so in particular a desingularization 
$W_i$ of $S/\sigma_i$ has $t_2(W_i)=0$.
From Theorem \ref{3-involutions} we have $t_2(S)=0$, viz.,
$S$ satisfies Bloch's conjecture. 
\end{proof} 

\begin{rk}
As noted in \cite[\S5]{BC1}, there is only one Burniat surface with
$K_S^2=2$, and it is a surface in the 6-dimensional family of
classical Campedelli surfaces with $\pi_1(S)=(\Z/2)^3$. Thus this
surface fits into the discussion in Section \ref{sec:K2=2}, 
where we also noted that such surfaces satisfy Bloch's conjecture.
\end{rk}

\medskip\goodbreak
{\it Surfaces with $K^2=6$} 

\medskip\noindent
We now consider a minimal surface $S$ of general type with $p_g(S)=0$ 
and $K^2_S=6$. Either the bicanonical map $\Phi_2:S\to\P^6$ is birational 
or the degree of $\Phi_2$ is  either 2 or 4, by \cite[Thm.1.1]{MP04}.

When the degree of $\Phi_2$ is 4 and $K^2_S=6$, a complete
classification has given in \cite{MP01}. The classification shows that
all these surfaces are Burniat surfaces. It follows that they satisfy
Bloch's conjecture, by Theorem \ref{thm:Burniat}.

Now suppose that the degree of $\Phi_2$ is 2. Then $\Phi_2$ determines
an involution $\tau$ on $S$, and $\Phi_2$ is composed with $\tau$.  By
\cite[1.2]{MP04} there is a fibration $f\colon S\to\P^1$, whose general
fibre $F$ is hyperelliptic of genus 3, such that the bicanonical
involution $\tau$ on $S$ induces the hyperelliptic involution on $F$
and $f$ has either $r=4$ or $r=5$ double fibres. (Both possibilities
occur; see \cite[\S4]{MP04}).

In particular $S/\tau$ is rational. Forming the square 
\eqref{square:S-sigma}, the desingularization $Y$ of $S/\tau$ is rational by
\cite[2.1]{MP04}. Hence $t_2(Y)=0$.  By Theorem \ref{sigma-composed}, 
$-4\le K_Y^2\le0$ and the fixed locus of $\tau$ has $10$ isolated points.
The cases of $r=4$ or $5$ double fibers corresponds to the cases 
$K^2_Y=-4$ and $K^2_Y>-4$.

\medskip 
Although we have been unable to show that
Bloch's conjecture holds for every surface with $K^2_S=6$ whose
bicanonical map has degree 2, we can show this for the examples
constructed in \cite[\S4]{MP04}. We shall use the notation of 
Section \ref{sec:K2=7},
since the construction of these examples are variations of the
construction of Inoue's surface described there.

\medskip
\begin{constr}\label{construct.1}
In the example in \cite[4.1]{MP04} we assume that the conics $F_1$,
$F_2$, $F_3$ all pass through a general point $P$ and that pairwise they
intersect transversally at $P$. In this case the branch locus 
$D=D_1+D_2+D_3$ acquires a singular point of type $(1,1,1)$,
and the resulting bidouble cover $X\to\Pi$ has a singularity over 
the image $P'$ of $P$. Blowing up $P'$ gives $\Pi'$; 
normalizing $X\times_{\Pi}\Pi'$ yields a bidouble cover $S \to \Pi'$, 
and the surface $S$ has $p_g=0$ and $K^2_S=6$. 
The surface $S$ has a fibration $f:S\to\P^1$, whose general fiber is 
hyperelliptic of genus 3 and $f$ has 4 double fibers. Moreover the 
exceptional divisor of $S \to X$ is a smooth rational curve $C$ with $C^2=-4$,
$$\CD
 S@>{p}>>\Pi' \\
 @VVV  @VVV \\ 
 X@>{p_0}>> \Pi \\ 
\endCD$$
\end{constr}

\begin{prop}\label{K2=6.1}
The surface $S$ of \ref{construct.1}
satisfies Bloch's conjecture.
\end{prop}

\begin{proof} 
We shall write $\sigma_i$ for the three nontrivial involutions on $X$
and on $S$ associated to the bidouble covers. We first compute the
number $k'_i$ of isolated fixed points on $X$ of $\sigma_i$, as we did
in Lemma \ref{k-Inoue}. Since $P \in D_1 \cap D_2 \cap D_3$ and 
$F_1\cdot F_3$ contains only 1 point outside of $P_1,\cdots\!,P_6$
we get $k'_1=(D_2\cdot D_3)=6$. Similarly we get $k'_2=8$ and $k'_3=4$.
Arguing again as in Lemma \ref{k-Inoue}, 
this implies that the number $k_i$ of isolated
fixed points of $\sigma_i$ on $S$ is: $k_1=10$, $k_2=k_3=8$.

By Theorem \ref{sigma-composed}, the bicanonical map $\Phi_2$ is
composed with $\sigma_1$ but not with $\sigma_2$ or $\sigma_3$. In
particular, $S/\sigma_1$ is rational. We will show that the
desingularizations $Y_2$ and $Y_3$ of $S/\sigma_2$ and $S/\sigma_3$
satisfy Bloch's conjecture, so $t_2(Y_2)=t_2(Y_3)=0$; 
Theorem \ref{3-involutions} will imply that $t_2(S)=0$, i.e., 
that $S$ satisfies Bloch's conjecture.

By symmetry it suffices to consider $Y_2$. Let $X_2$
denote the blowup of $S$ along the $k_2=8$ isolated fixed points of
$\sigma_2$, and form the square
\[
\CD X_2@>{h}>>S \\ 
@V{\bar \sigma_2}VV @V{\sigma_2}VV \\ 
Y_2@>{g}>> S/\sigma_2 \\ 
\endCD
\]
as in \eqref{square:S-sigma}. The images in $Y_2$ of the 8 exceptional
curves on $X_2$ are 8 disjoint nodal curves $C_1,\dots,C_8$. As
pointed after \eqref{square:S-sigma} the branch locus of $X_2\to Y_2$
is $B=g^*(D_2)+C+\sum C_j$. We get a minimal resolution $W$ of
$S/\sigma_2$ with branch locus $\bar B$ by blowing down curves on
$Y_2$. By Theorem \ref{not-composed}, $W$ cannot have Kodaira dimension 2:
$\bar B$ is not the disjoint union of 4 nodal curves,
and $k=8>K^2_S=6$.
Hence $W$ and $Y_2$ are not of general type, so $t_2(Y_2)=0$ as desired.
\end{proof}
 
\begin{constr}\label{construct.2}
The second example of a surface of general type with $p_g=0$,
$K^2_S=6$ is constructed in \cite[4.2]{MP04}; the fibration $S \to
\P^1$ has a hyperelliptic curve of genus 3 as its general fiber, and 5
double fibers. We start with the same configuration of 6 points
$P_1,\dots,P_6$ as in Section \ref{sec:K2=7}, and consider the point 
$P_7=\Delta_2\cap\Delta_3$. Form the blowup $\Pi'$ of $\pi$ at $P_7$, 
and write $e_7$ for the corresponding exceptional divisor. 
Let $\Delta'_2$ and $\Delta'_3$ denote the strict transform of 
$\Delta_2$ and $\Delta_3$ on $\Pi'$ and set $D=D_1+D_2+D_3$, where
\[
D_1=C +L_1 +L_2 \ ; \ D_2=F_3 \ ; \ D_3= F_1 +F'_1 + \Delta^{'}_2
+\bar \Delta^{'}_3 +L_3 +L_4.
\] 
Here $F_1,F'_1\in |F_1|$,  $F_3\in|F_3|$ and $C\in |F_2+F_3-2e_7|$ 
are general curves. As noted in \cite{MP04}, $C$ is a smooth irreducible 
curve of genus 0, and we get a
smooth bidouble cover $\pi:X\to\Pi'$ with $p_g(X)=0$.

For $i=1,\dots,4$ the inverse image of $L_i$ in $X$ is the disjoint
union of two $(-1)$-curves $E_{i1}, E_{i2}$. Also the inverse image of
$ \Delta^{'}_2$ is the disjoint union of two $(-1)$-curves $E_1$ and
$E_2$. The system $|-K_{\Pi'}|$ gives a degree 2 morphism $\Pi'\to\P^2$.
The surface $S$ obtained from $X$ by contracting $E_1,E_2$ and
$E_{ij}$, with $1\le i \le 4 $ and $1\le j \le 2$ is minimal of
general type with $K^2_S=6$ and $p_g(S)=0$.
\end{constr}

\begin{prop}\label{K2=6.2}
The surface $S$ of \ref{construct.2} satisfies Bloch's conjecture.
\end{prop}

\begin{proof} 
We shall write $\sigma_i$ for the three nontrivial involutions on $X$
and on $S$ associated to the bidouble covers. As noted in the proof of
Proposition \ref{K2=6.2}, the bicanonical map $\Phi_2$ is composed with
$\sigma_1$ and hence $S/\sigma_1$ is rational and $t_2(S/\sigma_1)=0$.

Next, we compute the number $k'_2$ and $k'_3$ of isolated fixed points
on $X$ of $\sigma_2$ and $\sigma_3$ and the corresponding number $k_2$
and $k_3$ of isolated fixed points on $S$, as we did in
Lemma \ref{k-Inoue}.

We have $k'_2=(D_1\cap D_3)=
(C\cdot F_1)+(C\cdot F'_1)+(C\cdot\Delta'_2) + (C\cdot\Delta'_3)$. 
Now $(C\cdot F_1)=(C\cdot F'_1)=4$ and 
$(C\cdot \Delta'_2)=(C\cdot\Delta'_3)=0$, because 
$e_7\cdot\Delta'_2= e_7 \cdot\Delta^{'}_3=1$. 
Therefore $k'_2=8$ and $k_2=k'_2=8$ because $D_2=F_3$
does not contain any of the $L_i$.

Similarly $k'_3= (D_1\cdot D_2)=(F_2+F_3)\cdot F_3=2$, because $(F_3)^2=0$.
This last equality follows from adjunction because $F_3$ has genus 0 and 
$K_{\Pi} \cdot F_3= (-\Delta_1 +\Delta_2 +\Delta_3)\cdot F_3= -2$. 
The fixed divisorial part $D_3$ of $\sigma_3$ contains $L_3 +L_4$ and
$\Delta^{'}_2$ which contract on $S$ to the 6 points corresponding to
$E_{ij}$ and to $ E_1,E_2$. Therefore $k_3=k'_3 +6= 8$.

Since $k_2=k_3=8$ are less than $K^2_S +4=10$, Theorem \ref{sigma-composed}
implies that $\Phi_2$ is not composed with either $\sigma_2$ or $\sigma_3$.
As in the previous example let $X_i$ be the blowup of $S$ along the fixed
points of $\sigma_i$, and let $\pi_i:X_i\to Y_i$ be the map to the
desingularization $Y_i$ of $S/\sigma_i$ ($i=2,3$). Then $\pi_i$ is
branched on $B_i= g^*_i(D_i) + \sum_{1\le h\le8} C^h_i$, with
$(C^h_i)^2= -2$. Because the image $\bar B_i$ of the branch locus
$B_i$ in a minimal model $W_i$ of $Y_i$ is not the disjoint union of 4
nodal curves, it follows from Theorem \ref{not-composed} that the 
Kodaira dimension of $W_i$ cannot be 2.  Therefore both $Y_2$ and $Y_3$ 
satisfy Bloch's conjecture, so $t_2(Y_2)=t_2(Y_3)=0$. 
By Theorem \ref{3-involutions}, this shows that $t_2(S)=0$ 
and finishes the proof.
\end{proof}

\begin{subrem} 
The same argument applies to the surface in Example 4.3 in
\cite{MP04}, showing that it also satisfies Bloch's conjecture.
\end{subrem}

\section*{Acknowledgements}

The authors would like to thank Rita Pardini and Carlos Rito for 
their assistence in understanding the details of the constructions 
of surfaces of general type. We also thank Fabrizio Catanese and 
Ingrid Bauer for many useful comments during the preparation of this paper.

\end{document}